\documentclass[12pt]{amsart}
\usepackage{amssymb,latexsym,comment,url}
\usepackage[all]{xy}

\newfont{\wncyr}{wncyr10 at 12pt}
\newfont{\wncyrten}{wncyr10 at 10pt}

\newcommand{\Gal}{{\operatorname{Gal}}}
\newcommand{\rank}{\operatorname{rank}}
\newcommand{\PGL}{\operatorname{PGL}}
\newcommand{\Br}{\operatorname{Br}}
\newcommand{\Ob}{\operatorname{Ob}}
\newcommand{\GL}{\operatorname{GL}}
\newcommand{\ra}{\longrightarrow}
\newcommand{\PP}{{\mathbb P}}
\newcommand{\kbar}{{\overline k}}
\newcommand{\F}{{\mathbb F}}
\newcommand{\R}{{\mathbb R}}
\newcommand{\Q}{{\mathbb Q}}
\newcommand{\Z}{{\mathbb Z}}
\newcommand{\isom}{\cong}
\newcommand{\Sha}{\mbox{\wncyr Sh}}

\newenvironment{ProofOf}[1]{\par\noindent{\em Proof of #1.}}%
                        {\hspace*{\fill}\nobreak$\Box$\par\medskip}

\newtheorem{Proposition}{Proposition}[section]
\newtheorem{Theorem}[Proposition]{Theorem}
\newtheorem{Lemma}[Proposition]{Lemma}

\theoremstyle{definition}
\newtheorem{Definition}[Proposition]{Definition}
\newtheorem{Remark}[Proposition]{Remark}
\newtheorem{Example}[Proposition]{Example}

\begin{document}

\title[Visible 2-torsion in the Tate-Shafarevich group]%
{Visible 2-torsion in the Tate-Shafarevich group \\ of an elliptic curve}

\author{Tom~Fisher}
\address{University of Cambridge,
         DPMMS, Centre for Mathematical Sciences,
         Wilberforce Road, Cambridge CB3 0WB, UK}
\email{T.A.Fisher@dpmms.cam.ac.uk}

\date{29th May 2026}  

\begin{abstract}
  We show that every pair of $2$-torsion elements in the
  Tate-Shafarevich group of an elliptic curve are visible in the same
  abelian surface. This was previously only known for a single
  $2$-torsion element.  Our result explains some of the original
  observations on visibility in the paper of Cremona and Mazur.
\end{abstract}

\maketitle

\renewcommand{\baselinestretch}{1.1}
\renewcommand{\arraystretch}{1.3}
\renewcommand{\theenumi}{\roman{enumi}}

\section{Introduction}

Let $E$ be an elliptic curve defined over a number field $k$.  A
torsor (or principal homogeneous space) under $E$ is a pair $(C,\mu)$
where $C$ is a smooth curve of genus one, and $\mu : E \times C \to C$
is morphism (both defined over $k$) inducing a simply transitive
action on $\kbar$-points.  The Weil-Ch\^atelet group $H^1(k,E)$
parametrises the torsors under $E$ up to $k$-isomorphism. The
Tate-Shafarevich group
\[ \Sha(E/k) = \ker \left( H^1(k,E) \to \prod_v H^1(k_v,E) \right) \]
is the subgroup of torsors that are everywhere locally soluble.  This
group is conjecturally finite, and when finite, Cassels \cite{CaIV}
showed it has order a square.  Its non-zero elements correspond to
genus one curves with Jacobian $E$ that are counterexamples to the
Hasse principle.  Mazur introduced the notion of visibility as a way
to study elements of $\Sha(E/k)$.

\begin{Definition}
  An element $\xi \in H^1(k,E)$ is {\em visible} in an abelian variety
  $A$ if there is an inclusion of abelian varieties $\iota : E \to A$
  such that $\xi$ is in the kernel of the induced map
  \[ \iota_* : H^1(k,E) \to H^1(k,A). \]
  Equivalently, $\xi$ is visible in $A$ if the corresponding torsor
  may be realised as a coset of $E$ inside $A$.
\end{Definition} 

Cremona and Mazur \cite{CM, AS2} gave examples, for each
$n \in \{2,3,4,5 \}$, of elliptic curves $E/\Q$ where all of
$\Sha(E/\Q) \isom (\Z/n\Z)^2$ is visible in the same abelian surface.
Mazur \cite{Mazur} showed that every 3-torsion element in $\Sha(E/k)$
is visible in an abelian surface. Klenke \cite{Klenke} showed that
every 2-torsion element in $H^1(k,E)$ is visible in an abelian
surface. An argument using restriction of scalars, first recorded by
Agashe and Stein \cite{AS}, shows that every $n$-torsion element in
$\Sha(E/k)$ is visible in an abelian variety of dimension $n$. Fisher
\cite{invis, vis7} gave examples of $6$ and $7$-torsion elements in
$\Sha(E/\Q)$ that are not visible in an abelian surface, and also of
$7$-torsion elements that are visible in an abelian threefold.

It follows from the restriction of scalars argument that $\Sha(E/k)$
is finite if and only if all its elements are visible in the same
abelian variety.

We prove the following.

\begin{Theorem}
  \label{thm:main}
  Every pair of $2$-torsion elements in $\Sha(E/k)$ are visible in the
  same abelian surface.
\end{Theorem}

The proof works by showing that certain complete intersections of two
quadrics in $\PP^5$ satisfy the smooth Hasse principle.

In Section~\ref{sec:2sel} we rephrase Theorem~\ref{thm:main} as a
statement about $2$-Selmer groups. In Section~\ref{sec:qi} we
introduce the relevant quadric intersections and describe some of
their geometry.  In Sections~\ref{sec:hp} and~\ref{sec:bs} we give two
different proofs that these satisfy the smooth Hasse principle. We
complete the proof of Theorem~\ref{thm:main} in
Section~\ref{sec:rel-vis}. In the final two sections we give some
examples and discuss a variant of the proof.

We thank Shiva Chidambaram for sharing some data in connection
with~\cite{BCC} which prompted us to revisit this problem.  A
Magma~\cite{magma} file checking some of the calculations and examples
in this paper is available on the author's webpage.

\section{The $2$-Selmer group}
\label{sec:2sel}

Let $E$ be an elliptic curve defined over a number field $k$, and let
$n \geqslant 2$ be an integer. We write $E[n]$ for the $n$-torsion
subgroup. The $n$-Selmer group
\[ S^{(n)}(E/k) = \ker \left( H^1(k,E[n]) \to \prod_v H^1(k_v,E) \right) \]
parametrises the $n$-coverings of $E$ that are everywhere locally
soluble. It is the middle term in the Kummer exact sequence
\[ 0 \ra E(k)/nE(k) \ra  S^{(n)}(E/k) \ra \Sha(E/k)[n] \ra 0. \]

\begin{Lemma}
\label{vislem}
Let $F/k$ be a second elliptic curve and $\pi: E[n] \to F[n]$ an
isomorphism of $\Gal(\kbar/k)$-modules. If
$\xi, \eta \in S^{(n)}(E/k)$ are in the kernel of the composition
\[ H^1(k,E[n]) \stackrel{\pi_*}{\ra} H^1(k,F[n]) \ra H^1(k,F) \] then
their images in $\Sha(E/k)$ are visible in the same abelian surface.
\end{Lemma}

\begin{proof} We identify $E[n]$ and $F[n]$ via $\pi$ and write
  $\Delta = E[n] = F[n]$.  We embed $\Delta$ diagonally in
  $E \times F$ and let $A = (E \times F)/\Delta$.  There is a
  commutative diagram
  \[ \xymatrix{ H^1(k,\Delta) \ar[d] \ar[r]
      & H^1(k,F) \ar[d] \\
      H^1(k,E) \ar[r] & H^1(k,A) \rlap{.} } \]
  We are given that $\xi$ and $\eta$ are in the kernel of the top
  horizontal map. By the commutativity of the diagram,
  their images in $\Sha(E/k)$ % \subset H^1(k,E)$
  are in the kernel of the bottom horizontal map, and so visible in $A$.
\end{proof}

We now take $n=2$. Elements of the $2$-Selmer group $S^{(2)}(E/k)$ may
be represented by binary quartics. See for example
\cite{AKM3P,BSDI,Cr,CF,Weil54}.

\begin{Definition}
A {\em binary quartic} is a homogeneous polynomial
$f \in k[x,z]$ of degree $4$. The invariants of the binary quartic
\begin{equation*}
f(x,z) = a x^4 + b x^3 z + c x^2 z^2 + d x z^3 + e z^4
\end{equation*}
are
\begin{equation}
\label{def:IJ}
\begin{aligned}
I &= 12 a e - 3 b d + c^2 , \\
J &= 72 a c e - 27 a d^2 - 27 b^2 e + 9 b c d - 2 c^3.
\end{aligned}
\end{equation}
A binary quartic is {\em non-singular} if it has no repeated roots
(i.e. linear factors) over $\kbar$,
equivalently if its discriminant $4 I^3 - J^2$ is non-zero.
\end{Definition}

A non-singular binary quartic $f$ defines a smooth projective curve of
genus one with affine equation $y^2 = f(x,1)$. The projective curve
is defined either by gluing affine pieces, or directly as the curve
with equation $y^2 = f(x,z)$ in the weighted projective plane
$\PP(1,2,1)$. This curve is a $2$-covering of its Jacobian
\begin{equation}
\label{jac}
 E_{I,J} : \quad y^2 = x^3 - 27 I x - 27 J. 
\end{equation}
A binary quartic is {\em soluble} if the corresponding genus one curve
has a $k$-point, and {\em everywhere locally soluble} if it has a
$k_v$-point for all places $v$.

\begin{Definition}
  Let $K/k$ be any field extension.  Binary quartics $f$ and $g$ are
  {\em properly $K$-equivalent} if
  \[ f(x,z) = \frac{1}{(p s - q r)^2} g (p x + r z, q x + s z) \]
  for some $p, q, r, s \in K$ with $p s - q r \not= 0$.
\end{Definition}

Binary quartics that are properly $k$-equivalent have the same
invariants. The converse is true for non-singular binary quartics over
an algebraically closed field.  We may write any elliptic curve $E$
defined over $k$ in the form $E_{I,J}$ for some $I,J \in k$ with
$4 I^3 - J^2 \not= 0$.  The $2$-Selmer group $S^{(2)}(E_{I,J}/k)$ may
then be identified with the set of proper $k$-equivalence classes of
everywhere locally soluble binary quartics with invariants $I$ and
$J$.

The genus one curve defined by a non-singular binary quartic $f$ is a
double cover of $\PP^1$ ramified over $4$ points.  These $4$ points
are the fibre of the $2$-covering above the identity on $E$,
and form a torsor under $E[2]$. In Section~\ref{sec:rel-vis} we will
use the fact that the class of $f$ in $H^1(k,E[2])$ is the class
of this torsor.

\section{Quadric intersections}
\label{sec:qi}

We work over a field $k$ with
$\operatorname{char}(k) \not= 2,3$, with algebraic closure $\kbar$.
Let $F(x_1, \ldots, x_n)$ and $G(x_1, \ldots, x_n)$ be a pair of
quadratic forms in $n$ variables. Let $A$ and $B$ be the $n \times n$
symmetric matrices of second partial derivatives of $F$ and $G$.  Then
$\det(\lambda A + \mu B)$ is a binary form of degree $n$.  By abuse of
notation we write this binary form as $\det(\lambda F + \mu G)$. The
next two lemmas are well known. The first is \cite[Proposition 2.1]{Reid}.

\begin{Lemma}
  \label{lem:qi-smooth}
  The quadric intersection $X = \{F = G = 0\} \subset \PP^{n-1}$ is
  smooth (and of codimension $2$) if and only if
  $\det(\lambda F + \mu G)$ has no repeated factors.
\end{Lemma}
\begin{proof}
  By the Jacobian criterion for smoothness, any singular point on $X$
  must be a singular point on some quadric in the pencil spanned by
  $F$ and $G$. It therefore gives a root of $\det(\lambda F + \mu G)$.
  This is a repeated root, as may be seen by reducing to the
  case where $F$ is a quadratic form in $x_2, \ldots, x_n$ only, and
  $G$ has no term $x_1^2$. Conversely, a root of the binary form is a
  repeated root if and only if the corresponding quadric has rank
  less than $n-1$, or it has rank exactly $n-1$ and its singular
  point lies on $X$. In both these cases $X$ has a singular point.
\end{proof}

\begin{Lemma}
  \label{lem:qi-jac}
  Let $X$ be as in Lemma~\ref{lem:qi-smooth} with $n=4$. Then $X$ is a
  smooth curve of genus one. If $\rank F = 3$ then $X$ is a
  $2$-covering of its Jacobian elliptic curve $y^2 = \det(xF + G)$.
\end{Lemma}

\begin{proof}
  This is a special case of a standard construction used in $4$-descent.
  See for example \cite{AKM3P} or \cite{4desc}.
\end{proof}

Given a binary quartic form
\begin{equation*}
f(x,z) = a x^4 + b x^3 z + c x^2 z^2 + d x z^3 + e z^4,
\end{equation*}
we define the {\em associated quadratic form}
\begin{equation*}
Q_f(x_0,x_1,x_2) = a x_0^2 + b x_0 x_1
  + \frac{c}{3} (x_0 x_2 + 2 x_1^2) + d x_1 x_2 + e x_2^2.
\end{equation*}
This choice is made so that
\[ Q_f(x^2,xz,z^2) = f(x,z) \]
and in addition the following two properties hold.

\begin{Lemma}
  \label{lem:first-props}
  Let $Q(x_0,x_1,x_2) = x_0 x_2 - x_1^2$. Let $f$ be a binary quartic
  form, with invariants $I$ and $J$ as defined by~\eqref{def:IJ}.
  \begin{enumerate} \item
    We have $\det(x Q - 9 Q_f) = 2(x^3 - 27 I x - 27 J)$.
  \item If $\widetilde{f}(x,z) = f(px + rz,qx+sz)$ then
    $Q_{\widetilde{f}}(x_0,x_1,x_2)
    =  Q_f(\widetilde{x}_0,\widetilde{x}_1,\widetilde{x}_2)$
    where
    \begin{align*}
    \widetilde{x}_0 &= p^2 x_0 + 2pr x_1 + r^2 x_2, \\
    \widetilde{x}_1 &=  pq x_0 + (ps +qr) x_1 + rs x_2, \\
    \widetilde{x}_2 &=  q^2 x_0 + 2qs x_1 + s^2 x_2.
    \end{align*}
    Moreover $ Q(\widetilde{x}_0,\widetilde{x}_1,\widetilde{x}_2)
    = (ps - qr)^2 Q(x_0,x_1,x_2)$.
\end{enumerate}
\end{Lemma}
\begin{proof}
  A direct calculation.
\end{proof}

\begin{Definition}
  \label{def:Xfg}
  Let $f$ and $g$ be a pair of non-singular binary quartics with
  the same invariants $I$ and $J$. The {\em associated quadric
  intersection} is 
\[ X_{f,g} = \left\{ \begin{aligned} u_0 u_2 - u_1^2 &= v_0 v_2 - v_1^2 \\
Q_f(u_0,u_1,u_2) &= Q_g(v_0,v_1,v_2) \end{aligned} \right\} \subset \PP^5. \]
\end{Definition}

In Section~\ref{sec:hp} we show that quadric intersections of the form
$X_{f,g}$ satisfy the smooth Hasse principle. First however we
consider their geometry.

By Lemma~\ref{lem:first-props}(i) the quadratic forms $F$ and $G$
defining $X_{f,g}$ satisfy
\begin{equation}
  \label{cubic-squared}
  \det(xF - 9G) =  -4(x^3 - 27 I x - 27J)^2.
\end{equation}
In particular the pencil of quadrics defining $X_{f,g}$ has 3 singular
fibres. Each is singular along a line of the form
\[  \{ (\lambda u_0 : \lambda u_1 : \lambda u_2 : \mu v_0 : \mu v_1 : \mu v_2 ) \mid (\lambda:\mu) \in \PP^1 \} \] 
for some $(u_0:u_1:u_2),(v_0:v_1:v_2) \in \PP^2$ with $u_0 u_2 - u_1^2 \not=0$
and $v_0 v_2 - v_1^2 \not=0$.
Such a line meets $X_{f,g}$ in 2 distinct
points swapped by the involution
\begin{equation}
\label{invol}
(u_0:u_1:u_2:v_0:v_1:v_2) \mapsto (u_0:u_1:u_2:-v_0:-v_1:-v_2).
\end{equation}
Therefore $X_{f,g}$ has exactly $6$ singular points.

\begin{Remark}
\label{rem:8planes}
We may change coordinates over $\kbar$ so that $X_{f,g}$ has equations
$\{ x_1 x_2 = x_3 x_4 = x_5 x_6 \} \subset \PP^5$.  It follows that
$X_{f,g}$ contains exactly $8$ planes.
\end{Remark}  

We recall from Section~\ref{sec:2sel} that binary quartics $f$ and $g$
are properly $\kbar$-equivalent if there exist $p,q,r,s \in \kbar$
with $ps-qr \not=0$ and
\[ f(x,z) = \frac{1}{(ps-qr)^2} g(px + rz,qx+sz). \]
This relation is unchanged if we multiply $p,q,r,s$ through by the
same non-zero scalar. It is natural to regard such equivalences as the
same.  With this convention, there are exactly $4$ proper
$\kbar$-equivalences relating $f$ and $g$, and these may be computed
by factoring a $(4,4)$-form as a product of $(1,1)$-forms, as
described in \cite[Theorem 12]{CF}.  Corresponding to each proper
equivalence there is a plane $\Lambda \subset X_{f,g}$ with equations
\begin{align*}
  (ps - qr) v_0 &= p^2 u_0 + 2pr u_1 + r^2 u_2, \\
  (ps - qr) v_1 &= pq u_0 + (ps +qr) u_1 + rs u_2,\\
  (ps - qr) v_2 &= q^2 u_0 + 2qs u_1 + s^2 u_2.
\end{align*}

\begin{Lemma}
\label{lem:2planes}
  The planes corresponding to distinct proper
  equivalences intersect only at singular points of $X_{f,g}$.
\end{Lemma}
\begin{proof}
  By Lemma~\ref{lem:first-props}(ii) we may reduce to the case where
  $f=g$. The proper equivalences then give the action of the
  $2$-torsion of the Jacobian.  We may take one to be given by the
  identity matrix, and the other by a matrix
  $(\begin{smallmatrix} p & q \\ r & s \end{smallmatrix})$.  Since the
  latter has order $2$ in $\PGL_2$, it has trace $0$.  A calculation
  then shows that the planes have unique point of intersection
  \begin{equation}
    \label{singpt}
    (u_0:u_1:u_2:v_0:v_1:v_2) = (-r : p : q : -r : p : q)
  \end{equation}
  and that this is a singular point on $X_{f,f} = X_{f,g}$.
\end{proof}

The next lemma shows that it is easy to find a $k$-rational conic on
$X_{f,g}$. By a conic, we shall always mean a smooth conic.

\begin{Lemma}
  \label{lem:getconic}
  Let $F = u_0u_2 - u_1^2 - v_0v_2 + v_1^2$ be the first of the two
  quadrics defining $X_{f,g}$. Then there are two algebraic families
  of planes contained in $\{ F = 0 \}$, each parametrised by $\PP^3$.
  The planes in each family that meet $X_{f,g}$ in a conic are
  parametrised by the complement of $4$ planes in $\PP^3$.
\end{Lemma}
\begin{proof}
  We change coordinates over $k$ so that
  $F = z_{12} z_{34} - z_{13} z_{24} + z_{14} z_{23}$, this being the
  formula for the Pfaffian of a $4 \times 4$ alternating matrix
  $Z$. The two algebraic families are the orbits of the planes
  $\{ z_{12} = z_{13} = z_{14} = 0\}$ and
  $\{z_{23} = z_{24} = z_{34} = 0\}$ under the action of $\GL_4$ via
  $A: Z \mapsto A Z A^T$. We may describe one of these families as
  follows.  Let $\Phi : \PP^3 \times \PP^3 \,\, - \to \, \PP^5$ be the
  rational map given by
  \begin{align*}
    ((x_1:x_2:x_3:&x_4),(y_1:y_2:y_3:y_4)) \\ &\mapsto (z_{12} :
    \ldots : z_{34} ) = (x_1 y_2 - x_2 y_1 : \ldots : x_3 y_4 - x_4 y_3).
  \end{align*}
  Then for each point $x \in \PP^3$ the image of $y \mapsto \Phi(x,y)$
  is a plane $\Pi_x \subset \{F = 0\}$.

  As suggested by Remark~\ref{rem:8planes} we now change coordinates
  over $\kbar$ so that $X_{f,g}$ is defined by
  $F = z_{12} z_{34} - z_{13} z_{24} + z_{14} z_{23}$ and
  $G = \lambda z_{12} z_{34} + \mu z_{13} z_{24}$ where
  $\lambda,\mu \in \kbar$ with $\lambda,\mu,\lambda+\mu \not= 0$.
  After substituting $z_{ij} = x_i y_j - x_j y_i$ into $G$, its matrix
  of second partial derivatives (with respect to $y_1, \ldots, y_4$) is
  \[ \begin{pmatrix}
      0 & \mu x_3 x_4 & \lambda x_2 x_4  &-(\lambda + \mu) x_2 x_3 \\
      \mu x_3 x_4 & 0 &-(\lambda+ \mu) x_1 x_4 & \lambda x_1 x_3 \\
      \lambda x_2 x_4 &-(\lambda + \mu) x_1 x_4 & 0 & \mu x_1 x_2 \\
      -(\lambda + \mu) x_2 x_3 & \lambda x_1 x_3 & \mu x_1 x_2 & 0
   \end{pmatrix}. \]
 The plane $\Pi_x$ meets $X_{f,g}$ in a conic if and only if this
 matrix has rank $3$. This is the case for all
 $x \in \PP^3 \setminus \{ x_1 x_2 x_3 x_4 = 0 \}$.
\end{proof}

\section{The Hasse principle}
\label{sec:hp}

We take $k$ a number field, and show that the quadric
intersections $X_{f,g}$ in Definition~\ref{def:Xfg} satisfy the smooth
Hasse principle. (The Hasse principle itself is not true.
See Example~\ref{ex:hpfails} for a counterexample.)

\begin{Theorem}
  \label{thm:hp}
  If $X_{f,g}$ has a smooth $k_v$-point for all places $v$ then it has
  a smooth $k$-point.
\end{Theorem}

We know by Lemma~\ref{lem:getconic} that $X_{f,g}$ contains a
$k$-rational conic.  So if $X_{f,g}$ were smooth then it would satisfy
the Hasse principle by a theorem of Salberger \cite{Sal}. We initially
tried to adapt his proof, to allow for the fact that $X_{f,g}$ is
singular. However, we found that a simpler argument is possible in our
situation.
In particular, we appeal to standard facts about del Pezzo surfaces,
instead of Salberger's results on conic bundle surfaces.

\begin{Remark}
  If $E_{I,J}$ has no $k$-rational $2$-torsion points
  then~\eqref{cubic-squared} shows that $X_{f,g}$ is a quadric
  intersection of type (F), as defined by Colliot Th\'el\`ene, Sansuc
  and Swinnerton-Dyer.  This case is specifically excluded from
  \cite[Theorem 3.20]{CSS}.  According to \cite[Remark 3.20.1]{CSS} it
  is not known in general whether quadric intersections of this type
  have trivial Brauer-Manin obstruction.
\end{Remark}

We have found two different proofs of Theorem~\ref{thm:hp}.  We give
details of both, since the first motivates the second.  Readers
uninterested in this motivation may %at this point
skip to the start of the next section.

We start our first proof of Theorem~\ref{thm:hp} by showing that every
smooth hyperplane section of $X_{f,g}$ satisfies the Hasse principle.

\begin{Lemma}
  \label{lem:slice}
  Let $H \subset \PP^5$ be a $k$-rational hyperplane. If
  $X_{f,g} \cap H$ is a smooth surface then it satisfies the Hasse
  principle.
\end{Lemma}

\begin{proof}
  We are given that $Y = X_{f,g} \cap H$ is a smooth intersection of
  two quadrics in $\PP^4$. It is therefore a del Pezzo surface of
  degree $4$. As explained in Section~\ref{sec:qi} the proper
  equivalences between $f$ and $g$ determine planes
  $\Lambda_1, \ldots, \Lambda_4 \subset X_{f,g}$. We cannot have
  $\Lambda_i \subset H$ as this would contradict that $Y$ is
  smooth. Intersecting with $H$ we therefore get lines
  $\ell_1, \ldots, \ell_4 \subset Y$.  These lines are jointly defined
  over $k$, by which we mean that they are permuted by the action of
  $\Gal(\kbar/k)$.  By Lemma~\ref{lem:2planes} they are disjoint.  We
  may therefore blow them down to give a del Pezzo surface of
  degree~$8$. It is well known that del Pezzo surfaces of degree~$8$
  satisfy the Hasse principle. This completes the proof.
\end{proof}

\begin{Remark} Applying the involution~\eqref{invol} to the planes
  $\Lambda_1, \dots, \Lambda_4$ gives a further 4 planes on $X_{f,g}$
  and hence a further $4$ lines $\ell'_1, \ldots, \ell'_4$ on the del
  Pezzo surface $Y$.  It may be checked using Remark~\ref{rem:8planes}
  that $\ell_i$ meets $\ell'_j$ if and only if $i \not= j$.  The two
  sets of four lines therefore constitute a ``double-four'' in the
  terminology of Swinnerton-Dyer \cite[Section 7]{SD}.  It follows
  that the del Pezzo surface of degree~$8$ (considered at the end of
  the last proof) has no $-1$-curves, and so is isomorphic over
  $\kbar$ to $\PP^1 \times \PP^1$ (instead of $\PP^2$ blown up at one
  point).
\end{Remark}

We could use Bertini's theorem to choose the hyperplane $H$ in
Lemma~\ref{lem:slice}. However the following more explicit lemma
simplifies the next steps.

\begin{Lemma}
  \label{lem:make-smooth}
  Let $\Pi \subset \{ u_0 u_2 - u_1^2 = v_0 v_2 - v_2^2 \} \subset \PP^5$
  be a $k$-rational plane. Let $H \subset \PP^5$ be a $k$-rational
  hyperplane that contains $\Pi$, but does not contain any of the
  singular points on $X_{f,g}$.  Then the surface $X_{f,g} \cap H$
  is either smooth or has a $k$-rational singular point.
\end{Lemma}

\begin{proof}
  The pencil of quadrics defining $X_{f,g}$ contains
  $F = u_0 u_2 - u_1^2 - v_0 v_2 + v_2^2$ of rank 6, and three
  singular quadrics $G_1,G_2,G_3$ each of rank 4. We claim that on
  restricting to the hyperplane $H$, each of these quadrics has rank
  4. Indeed the restriction of $F$ has rank at most $4$ since
  $\Pi \subset \{F = 0\}$, and rank at least $4$ since the rank can
  drop by at most $2$.  If the rank of $G_i$ dropped then $H$ would
  have to contain the singular locus of $G_i$. However this singular
  locus is a line meeting $X_{f,g}$ in two singular points.  The claim
  follows since we assumed that $H$ does not contain any of the
  singular points on $X_{f,g}$.
  
  The singular quadrics in the pencil defining $X_{f,g} \cap H$
  correspond to the roots of a binary quintic form. We may account for
  $4$ of the roots as corresponding to the restrictions of
  $F,G_1,G_2,G_3$. If any of these are repeated roots then the
  corresponding rank $4$ quadric is defined over $k$, and its singular
  point is a $k$-rational singular point on $X_{f,g} \cap H$.  The
  same argument applies when the binary quintic form is identically
  zero (a case also covered by \cite[Lemma 1.14]{CSS}), provided we
  select the restriction of $F$ as our rank $4$ quadric, the problem
  with $G_1,G_2,G_3$ being that these might not be defined over $k$.
  In the remaining case the binary quintic form has distinct roots, in
  which case $X_{f,g} \cap H$ is smooth by Lemma~\ref{lem:qi-smooth}.
\end{proof}

The following lemma will help us choose the hyperplane $H$ in
Lemma~\ref{lem:make-smooth} so that $X_{f,g} \cap H$ is everywhere
locally soluble.

\begin{Lemma}
\label{lem:prepare}
Let $X = \{F = G = 0\} \subset \PP^5$ be a quadric
intersection. Suppose that $\rank F \geqslant 5$ and
$\Pi \subset \{F = 0\} \subset \PP^5$ is a plane. Then for every
$P \in X_{\rm sm} \setminus \Pi$ there is a hyperplane $H$ containing
both $\Pi$ and $P$ such that $P$ is a smooth point on the intersection
$X \cap H$.
\end{Lemma}

\begin{proof}
  We change coordinates so that $\Pi = \{ (*:*:*:0:0:0) \}$ and
  $P = (0:0:0:1:0:0)$. A hyperplane of the required form is given by
  $H = \{\lambda x_5 + \mu x_6 = 0\}$, provided that the linear form
  $\lambda x_5 + \mu x_6$ does not vanish on $T_PX$. So we are done
  unless $T_P X = \{x_5 = x_6 = 0\}$. However in that case the first
  $4$ partial derivatives of $F$ vanish at $P$. This implies that $F$
  vanishes on the linear span of $\Pi$ and $P$, and so
  $\rank F \leqslant 4$. This contradicts our assumption that
  $\rank F \geqslant 5$.
\end{proof}

The remainder of the proof is based on arguments
in~\cite[Section~2]{Sal}.

\medskip

\begin{ProofOf}{Theorem~\ref{thm:hp}}
  Let $X_{f,g} \subset \PP^5$ be as given in Definition~\ref{def:Xfg}.
  By Lemma~\ref{lem:getconic} we may pick a $k$-rational plane
  $\Pi \subset \{ u_0 u_2 - u_1^2 = v_0 v_2 - v_2^2 \}$ such that
  $C = X_{f,g} \cap \Pi$ is a conic.  Let $S$ be the finite set of
  places $v$ with $C(k_v) = \emptyset$.
  
  For each place $v \in S$, we use Lemma~\ref{lem:prepare}, and our
  assumption that $X_{f,g}$ has a smooth $k_v$-point, to pick a
  $k_v$-rational hyperplane $H_v$ containing $\Pi$ and meeting
  $X_{f,g}$ transversely at a $k_v$-point.  The implicit function
  theorem (specifically \cite[Proposition 6.2]{CSS}) shows that there
  is a $v$-adic neighbourhood $U_v$ of $H_v$ in $(\PP^5)^\vee(k_v)$
  such that every hyperplane in $U_v$ meets $X_{f,g}$ transversely at
  a $k_v$-point.  Since $C$ is smooth, none of the singular points of
  $X_{f,g}$ are contained in $\Pi$. The hyperplanes containing $\Pi$,
  but not containing any of the $6$ singular points on $X_{f,g}$, form
  a Zariski open subset $W \subset \PP^2$ (the complement of at most
  $6$ lines).  By weak approximation we may pick a hyperplane $H$ in
  $W(k)$ that belongs to $U_v$ for each $v \in S$. We claim that
  $(X_{f,g} \cap H)(k_v) \not= \emptyset$ for all places $v$.  Indeed,
  this is true for places $v \not\in S$ since
  $C \subset X_{f,g} \cap H$ and it is true for places $v \in S$ since
  $H \in U_v$.

  We know by Lemma~\ref{lem:make-smooth} that $X_{f,g} \cap H$ is
  either smooth or has a $k$-point. It follows by
  Lemma~\ref{lem:slice} that $X_{f,g} \cap H$ has a $k$-point. By our
  choice of $H$, this is not a singular point on $X_{f,g}$.
\end{ProofOf}

\begin{Remark}
\label{rem:density}
The conclusion of Theorem~\ref{thm:hp} may be strengthened to say that
the $k$-points on $X_{f,g}$ are Zariski dense.  Indeed, in
Lemma~\ref{lem:slice} the corresponding strengthening follows from
standard facts about del Pezzo surfaces, and in the proof just given
there are clearly infinitely many choices for the hyperplane $H$ that
satisfy the required hypotheses.
\end{Remark}

\section{An alternative proof}
\label{sec:bs}

In the previous section we proved that the quadratic intersections
$X_{f,g}$ satisfy the smooth Hasse principle. Our proof worked by
considering hyperplane sections which were del Pezzo surfaces. However
it turns out that a more direct approach is possible.

\begin{Proposition}
  \label{prop:bs}
  Let $X = X_{f,g}$ be as in Definition~\ref{def:Xfg}. Let
  $\Lambda_1, \ldots, \Lambda_4$ be the planes contained in $X$
  corresponding to the proper equivalences between $f$ and $g$, and
  let $H$ be the hyperplane section. Then the linear system
  $|3 H - \Lambda_1 - \Lambda_2 - \Lambda_3 - \Lambda_4|$ defines a
  birational map $\phi : X \, - \!  \to Y$ where $Y \subset \PP^9$ is
  a Brauer-Severi threefold. Moreover, the birational map $\phi$ is
  regular at all smooth points of $X$.
\end{Proposition}

\begin{proof}
  The planes $\Lambda_1, \ldots, \Lambda_4$ are jointly defined over
  $k$, by which we mean that they are permuted by the action of
  $\Gal(\kbar/k)$. Having made this observation, we are free for the
  remainder of the proof to work over $\kbar$.  In particular, we make
  the change of coordinates suggested in Remark~\ref{rem:8planes}.

  We label our coordinates on $\PP^5$ as $x_I = x_{ij}$ where
  $I = \{i,j\}$ runs over the subsets of $\{1,2,3,4\}$ of size 2. By
  Lemma~\ref{lem:2planes} we may assume that
  \[ X = \{ x_{12}x_{34} = x_{13}x_{24} = x_{14}x_{23} \} \subset
    \PP^5 \] and each $\Lambda_i$ is defined by the vanishing of the
  $x_I$ with $i \in I$.  The space of cubic forms vanishing on
  $\Lambda_1 \cup \ldots \cup \Lambda_4$ is spanned by the monomials
  $x_I x_J x_K$ with $\#(I \cup J \cup K) = 4$.  Inside this space, a
  complement to the space of cubic forms vanishing on $X$ is spanned
  by the entries on and above the leading diagonal of the following
  symmetric matrix, where for simplicity we put
  $\alpha = x_{12} x_{34}$.
 \[ \begin{pmatrix}
    x_{12}x_{13}x_{14} & \alpha x_{12} & \alpha x_{13} & \alpha x_{14} \\
    \alpha x_{12} & x_{12}x_{23}x_{24} & \alpha x_{23} & \alpha x_{24} \\
    \alpha x_{13} & \alpha x_{23} & x_{13}x_{23}x_{34} &  \alpha x_{34} \\
    \alpha x_{14} & \alpha x_{24} & \alpha x_{34} & x_{14}x_{24}x_{34}
  \end{pmatrix} \]
  This matrix defines a rational map $\phi : X \, - \!  \to Y$ where
  $Y \subset \PP^9$ is the locus of all $4 \times 4$ symmetric matrices
  of rank $1$.  Let $\psi : \PP^3 \, - \!  \to X$ be given by
 \begin{equation}
 \label{ratparam}
   (x_1: x_2:x_3:x_4) \mapsto (x_{12} : \ldots : x_{34})
   = (x_1x_2 : \ldots : x_3 x_4).
 \end{equation}
 The composite $\phi \circ \psi : \PP^3 \to Y$ is the $2$-uple
 embedding, which is clearly an isomorphism. It follows that $\phi$ is
 birational and $Y$ is a Brauer-Severi variety. The final statement is
 checked by direct calculation.
\end{proof}

Our alternative proof of Theorem~\ref{thm:hp} is immediate from
Proposition~\ref{prop:bs} and the Hasse principle for Brauer-Severi
varieties. Remark~\ref{rem:density} is also immediate.

\section{Application to visibility}
\label{sec:rel-vis}

The results of Section~\ref{sec:2sel} reduce the proof of
Theorem~\ref{thm:main} to showing that if $f$ and $g$ are non-singular
binary quartics with the same invariants $I$ and $J$, and both are
everywhere locally soluble, then their classes $\xi$ and $\eta$ in
$S^{(2)}(E_{I,J}/k)$ satisfy the hypotheses of Lemma~\ref{vislem}.

\begin{Lemma}
  \label{lem:loc-sol}
  Let $f$ and $g$ be non-singular binary quartics over $k$ with the
  same invariants $I$ and $J$. If both $f$ and $g$ are everywhere
  locally soluble then $X_{f,g}$ has a smooth $k_v$-point for all
  places $v$.
\end{Lemma}

\begin{proof}
  By assumption there exist $x,y,z,X,Y,Z \in k_v$ with $y^2 = f(x,z)$
  and $Y^2 = g(X,Z)$. Since any smooth curve with a $k_v$-point has
  infinitely many $k_v$-points we may further assume that $yY \not=0$.
  Then $X_{f,g}$ has $k_v$-point
  \[(u_0:u_1:u_2:v_0:v_1:v_2) = (x^2 Y:xz Y:z^2Y : X^2 y : X Z y : Z^2
    y).\] This is a smooth point since, as noted in Section~\ref{sec:qi},
  the singular points on $X_{f,g}$ all have $u_0u_2 - u_1^2 \not= 0$.
\end{proof}

Lemma~\ref{lem:loc-sol} and Theorem~\ref{thm:hp} together show that
$X_{f,g}(k) \not= \emptyset$.  The proof of Theorem~\ref{thm:main} is
completed by the following proposition.  In the case that $E$ has a rational
$2$-torsion point, we also appeal to Remark~\ref{rem:density}.

\begin{Proposition}
  \label{prop:use-pt}
  Let $f$ and $g$ be non-singular binary quartics over $k$ with the
  same invariants $I$ and $J$.  Let $E = E_{I,J}$ be given
  by~\eqref{jac}.  Suppose that either
  \begin{enumerate}
  \item $X_{f,g}(k) \not= \emptyset$ and $E(k)[2] = 0$, or 
  \item the $k$-points on $X_{f,g}$ are Zariski dense. % in $X_{f,g}$.
  \end{enumerate}
  Then there exists an elliptic curve $F$ over $k$ and
  $\pi : E[2] \to F[2]$ an isomorphism of $\Gal(\kbar/k)$-modules such
  that the classes of $f$ and $g$ are in the kernel of the composition
  \begin{equation}
    \label{comp}
    H^1(k,E[2]) \stackrel{\pi_*}{\ra} H^1(k,F[2]) \ra H^1(k,F).
    \end{equation}
  \end{Proposition}

\begin{proof}
  Let $(u_0:u_1:u_2:v_0:v_1:v_2)$ be a $k$-point on $X_{f,g}$. We put
  \begin{equation}
    \label{def:lamu}
  \lambda = u_0 u_2 - u_1^2 \quad \text{ and } \quad
    \mu = Q_f(u_0,u_1,u_2).
  \end{equation}
  If $\lambda = \mu = 0$ then either $f$ or $g$ has a $k$-rational
  root (i.e. linear factor) and so is trivial in $H^1(k,E[2])$. We
  comment on this case at the end of the proof. So for now we may
  suppose that
  $(\lambda :\mu) \in \PP^1(k)$. We pick
  $\widetilde{\lambda}, \widetilde{\mu} \in k$ with
  $\lambda \widetilde{\mu} - \mu \widetilde{\lambda} = 1$.

  We consider the quadric intersection
  \[ C_f = \left\{ \begin{aligned} \mu (x_0x_2 - x_1^2) - \lambda
        Q_f(x_0,x_1,x_2) &= 0 \\ \widetilde{\mu} (x_0x_2 - x_1^2) -
        \widetilde{\lambda} Q_f(x_0,x_1,x_2) &= y^2 \end{aligned}
    \right\} \subset \PP^3. \] If $E(k)[2]=0$ then by
  Lemma~\ref{lem:first-props}(i) and~\eqref{jac}
  the first quadric has rank $3$. If
  $E(k)[2] \not= 0$ then we use our additional hypothesis (ii) to
  adjust our choice of $k$-point so that the same conclusion holds.
  In other words, our threefold $X_{f,g}$ is fibred over the modular
  curve $X_E(2) \isom \PP^1$ (which parametrises elliptic curves
  $2$-congruent to $E$) and we want to make sure that our $k$-point is
  not above a cusp.
  
  Lemmas~\ref{lem:qi-jac} and~\ref{lem:first-props}(i) now show that
  $C_f$ is a $2$-covering of the elliptic curve
  \begin{equation}
    \label{ellF}
    F : \quad y^2 = -(9\mu x + 9 \widetilde{\mu})^3
    + 27 I (9\mu x + 9 \widetilde{\mu}) (\lambda x + \widetilde{\lambda})^2
    + 27 J  (\lambda x + \widetilde{\lambda})^3.
  \end{equation}
  An explicit isomorphism of $\Gal(\kbar/k)$-modules
  $\pi : E[2] \to F[2]$ is given, relative to the Weierstrass
  equations~\eqref{jac} and~\eqref{ellF}, by
 \begin{equation}
   \label{pi}
   (x,0) \mapsto (-(\widetilde{\lambda} x - 9
   \widetilde{\mu})/(\lambda x - 9 \mu),0).
  \end{equation}
 
  The first quadratic form in the definition of $C_f$ defines a conic
  with $k$-rational point $(x_0:x_1:x_2) = (u_0:u_1:u_2)$.  By
  parametrising this conic, we may rewrite $C_f$ as a double cover of
  $\PP^1$. The set of ramification points, as a set with
  $\Gal(\kbar/k)$-action, is isomorphic to
  \begin{equation}
   \label{another:qi}
   \{ x_0x_2 - x_1^2 = Q_f(x_0,x_1,x_2) = 0 \} \subset \PP^2,
 \end{equation}  
 and hence also isomorphic to the set of ramification points for
 $y^2 = f(x,z)$.  We upgrade this to an isomorphism of torsors, when
 $E[2]$ and $F[2]$ are identified via $\pi$. To do this we first note
 that $E[2] \setminus \{0\}$ and $F[2] \setminus \{0\}$ are naturally
 in bijection with the singular fibres of the pencil of quadrics
 defining~\eqref{another:qi}.  This gives a more conceptual
 explanation for~\eqref{pi}.  Next, since each singular fibre is a
 pair of lines, each element of
 $E[2] \setminus \{0\} = F[2] \setminus \{0\}$ naturally determines a
 partition of the $4$ points~\eqref{another:qi} into 2 subsets of size
 2.  This determines the torsor structure.

 Our discussion shows that the induced map
 $\pi_* : H^1(k,E[2]) \to H^1(k,F[2])$ sends the class of
 $y^2 = f(x,z)$ (a 2-covering of $E$) to the class of $C_f$ (a
 2-covering of $F$). The latter has trivial image in $H^1(k,F)$ since
 $C_f$ has $k$-rational point $(x_0:x_1:x_2:y) = (u_0:u_1:u_2:1)$.
 The class of $f$ is therefore in the kernel of the
 composition~\eqref{comp}.

  The proof is completed by repeating the same arguments for the
  binary quartic $g$.  In view of the definition of $X_{f,g}$ we may
  rewrite~\eqref{def:lamu} as
  \[ \lambda = v_0 v_2 - v_1^2 \quad \text{ and } \quad
   \mu = Q_g(v_0,v_1,v_2). \]
  We then obtain a $2$-covering
  \[ C_g = \left\{ \begin{aligned} \mu (x_0x_2 - x_1^2) - \lambda
        Q_g(x_0,x_1,x_2) &= 0 \\ \widetilde{\mu} (x_0x_2 - x_1^2) -
        \widetilde{\lambda} Q_g(x_0,x_1,x_2) &= y^2 \end{aligned}
    \right\} \subset \PP^3 \] of the {\em same} elliptic curve $F$. The
  isomorphism $\pi : E[2] \to F[2]$ is given by~\eqref{pi} and so is
  also unchanged.  Since $C_g$ has $k$-rational point
  $(x_0:x_1:x_2:y) = (v_0:v_1:v_2:1)$, its class in $H^1(k,F[2])$
  has trivial image in $H^1(k,F)$.  The class of $g$ is therefore in the
  kernel of the composition~\eqref{comp}.

  Finally we remark that if one of the binary quartics (say $g$)
  is trivial in $H^1(k,E[2])$, then the above argument is only needed
  for the other one (say $f$). For this we start (following what
  is essentially the proof of Klenke~\cite{Klenke}) with any
  sufficiently general $k$-point $(u_0:u_1:u_2)$ on $\PP^2$ in place
  of a $k$-point on $X_{f,g}$. This avoids the problem that the
  latter could have $u_0 = u_1= u_2 = 0$.
\end{proof}

\section{Examples}
\label{sec:examples}

We give an example to illustrate the proof of Theorem~\ref{thm:main}.

\begin{Example} Let $E/\Q$ be the elliptic curve
\[  y^2 + y = x^3 - x^2 - 929 x - 10595 \]
labelled $571a1$ in \cite{CrTables}.
A $2$-descent shows that $S^{(2)}(E/\Q) \isom (\Z/2\Z)^2$, and its
non-zero elements are represented by
\begin{align*}
f(x,z) &= -11 x^4 + 68 x^3 z - 52 x^2 z^2 - 164 x z^3 - 64 z^4, \\
g(x,z) &= -4 x^4 - 60 x^3 z - 232 x^2 z^2 - 52 x z^3 - 3 z^4, \\
h(x,z) &= -31 x^4 - 78 x^3 z + 32 x^2 z^2 + 102 x z^3 - 53 z^4.
\end{align*}
Each of these binary quartics has invariants $I=44608$ and $J=18842960$.
In fact it may be shown (see \cite{GJPST}) that $E(\Q) = 0$ and
$\Sha(E/\Q) \isom (\Z/2\Z)^2$.

On the quadric intersection $X_{f,g} \subset \PP^5$ we find the point
\[ (u_0 : u_1 : u_2 : v_0 : v_1 : v_2) = (3 : 1 : 0 : -1 : 0 : 1). \]
Following the proof of Proposition~\ref{prop:use-pt}, with
$(\lambda,\mu,\widetilde{\lambda},\widetilde{\mu}) = (-1,211/3,0,-1)$,
shows that all of $\Sha(E/\Q)$ is visible
in an abelian surface isogenous to $E \times F$ where $F/\Q$ is
the elliptic curve
\[ y^2 + y = x^3 + x^2 - 4 x + 2 \]
labelled $571b1$ in \cite{CrTables}. This is the first pair of
elliptic curves in \cite[Table 2]{CM}.
\end{Example}

We now give examples to show that our proof of Theorem~\ref{thm:main}
for elements of $\Sha(E/\Q)[2]$ does not extend to elements of
$H^1(\Q,E)[2]$. In other words, we really do need the hypothesis in
Lemma~\ref{lem:loc-sol} that both binary quartics are everywhere
locally soluble.

\begin{Example}
  \label{ex:hpfails}
  Let $E/\Q$ be the elliptic curve $y^2 = (x-1)(x-3)(x+3)$ 
  labelled $192a2$ in \cite{CrTables}.
  The non-singular binary quartics
  \begin{align*}
    f(x,z) &= -x^4 - 4 x^2 z^2 - z^4, \\
    g(x,z) &= x^4 + 2 x^3 z + 2 x^2 z^2 + 4 x z^3 + 4 z^4,
  \end{align*} have the same invariants $I = 28$ and $J = -160$,
  and represent $2$-coverings of $E$.
  The associated quadric intersection is 
  \[ X_{f,g} = \left\{ \begin{aligned} 
    (u_0 + u_2)^2 + 2 u_1^2 + (v_0 + v_1)^2 + (v_1 + 2 v_2)^2 &= 0\\
    (u_0 - u_2)^2 + 6 u_1^2 + (v_0+v_1+2 v_2)^2 - 3 v_1^2 &= 0
                       \end{aligned} \right\} \subset \PP^5. \]
The $6$ singular points on $X_{f,g}$ are
$(\sqrt{2} : 0 : -\sqrt{2} : 2 : -2 : 1)$,
$(\sqrt{-2} : 0 : \sqrt{-2} : -2 : 0 : 1 )$,
$(0 : \sqrt{-1} : 0 : 2 : -1 : 1)$
and their Galois conjugates.
For $p \geqslant 5$ a prime, the genus one curves
$y^2 = f(x,z)$ and $y^2 = g(x,z)$
have good reduction, and so are soluble over $\Q_p$.
It follows by Lemma~\ref{lem:loc-sol}
that $X_{f,g}$ has a smooth $\Q_p$-point. There is also a smooth
$\Q_2$-point given by \[ (u_0:u_1:u_2:v_0:v_1:v_2)
  = ( \alpha^2 - 1: 1: 0: \alpha^2 - 1: \alpha: 1)\]
where $\alpha \in \Z_2$ is a root of $x^4 + x^3 - x^2 + x + 4 = 0$.
However, our first and second equations for $X_{f,g}$ show that
$X_{f,g}(\R)$ and $X_{f,g}(\Q_3)$ consist only of
the singular points defined over $\Q(\sqrt{2})$ and $\Q(\sqrt{-2})$
respectively.

In conclusion $X_{f,g}$ is a counterexample to the Hasse principle,
but not the smooth Hasse principle.
\end{Example}

Finally we give some examples with $X_{f,g}(\Q_p) = \emptyset$.
\begin{Proposition} Let $p \geqslant 7$ be
  a prime with $p \equiv 2,3 \pmod{5}$. The binary
  quartics 
\begin{align*}
  f(x,z) &= 20 x^2 z^2 + p (x^4 + 30 x^2 z^2 + 5 z^4),  \\
  g(x,z) &= (x^2 - 5 z^2)^2 + p (x^4 + 4 x^3 z - 30 x^2 z^2
    + 20 x z^3 + 25 z^4),
\end{align*}
are non-singular and have the same invariants. However, the associated
quadric intersection $X_{f,g} \subset \PP^5$ has no $\Q_p$-points.
\end{Proposition}
\begin{proof}
  Both $f$ and $g$ have invariants
  \begin{align*}
 I &= 2^4 \cdot 5 \cdot(12 p^2 + 15 p + 5), \\
     J &= -2^6 \cdot 5^2 \cdot(3 p + 2) (9 p^2 + 15 p + 5),
  \end{align*}
  and discriminant
  \[ 4 I^3 - J^2 = 2^{12} \cdot 3^3 \cdot 5^3 \cdot p^2
    (11 p^2 + 15 p + 5)^2 \not = 0. \]
  One of the quadrics in the pencil defining $X_{f,g} \subset \PP^5$ reduces
  mod $p$ to $20 u_1^2 - (v_0 - 5 v_2)^2$. So if
  $(u_0:u_1:u_2:v_0:v_1:v_2)$ is a $\Q_p$-point on $X_{f,g}$, with all
  coordinates in $\Z_p$, but not all in $p\Z_p$, then (since
  $5$ is a quadratic nonresidue mod $p$) we must have
  $u_1 \equiv v_0 - 5 v_2 \equiv 0 \pmod{p}$. We may
  therefore write
  \begin{equation}
  \label{subst}
  (u_0,u_1,u_2,v_0,v_1,v_2) = (x_1, p x_5, x_2, 5 x_4 + p x_6, x_3, x_4)
  \end{equation}
  for some $x_1, \ldots, x_6 \in \Z_p$. These then satisfy
  \begin{align*}
      x_1 x_2 + x_3^2 - 5 x_4^2 &\equiv 0 \pmod{p}, \\
    x_1^2 + 10 x_1 x_2 + 5 x_2^2 + 20 x_3^2 - 40 x_3 x_4 &\equiv 0 \pmod{p}.
  \end{align*}
  Since $p \equiv 2,3 \pmod{5}$ we have $[\F_p(\zeta):\F_p] = 4$,
  where $\zeta$ is a primitive $5$th root of unity. We put $\sqrt{5} =
  1 + 2(\zeta + \zeta^4)$, and write $\widetilde{x}_i$ for $x_i \pmod{p}$.
  Taking a linear
  combination of the last two equations gives 
  \[  (\widetilde{x}_1 + \sqrt{5} \widetilde{x}_2)^2
    = (\zeta - \zeta^4)^2 (2 \widetilde{x}_3
    - (5 - \sqrt{5}) \widetilde{x}_4)^2. \]
  It follows that $x_1, \ldots, x_4$ all vanish mod $p$. Combined
  with~\eqref{subst} this gives that $u_0, u_1,u_2,v_0,v_1, v_2$ all vanish
  mod $p$, contradicting our earlier assumption.
 \end{proof}

\section{Relation to the obstruction map}
\label{sec:ob}

The obstruction map $\Ob_n : H^1(k,E[n]) \to \Br(k)$ is defined in
\cite{CFOSS, ONeil}.  Taking $n=2$, the significance of this map is
that a $2$-covering $\xi \in H^1(k,E[2])$ may be represented by a
binary quartic if and only if $\Ob_2(\xi) = 0$.

Let $\xi,\eta \in H^1(k,E[2])$ be represented by binary quartics $f$
and $g$.  We saw in Section~\ref{sec:bs} that $X_{f,g}$ is birational
to a Brauer-Severi variety $Y$. It is natural to ask: what is the
class of $Y$ in $\Br(k)$? We guess it should be $\Ob_2(\xi + \eta)$.
This suggested the following result, which gives an alternative (more
computational) proof of Theorem~\ref{thm:main}.  We work over a field
$k$ of characteristic $0$.

\begin{Proposition}
\label{prop:ob}
Let $E/k$ be an elliptic curve.  Let $\xi,\eta \in H^1(k,E[2])$ with
$\Ob_2(\xi) = \Ob_2(\eta) = 0$.  The following statements are
equivalent.
\begin{enumerate}
\item There exists an elliptic curve $F/k$, and $\pi: E[2] \to F[2]$
  an isomorphism of $\Gal(\kbar/k)$-modules, such that $\xi$ and
  $\eta$ are in the kernel of the composition
\[ H^1(k,E[2]) \stackrel{\pi_*}{\ra} H^1(k,F[2]) \ra H^1(k,F). \]
\item  $\Ob_2(\xi + \eta) = 0$.
\end{enumerate}
\end{Proposition}

\begin{proof}
  ``(i) $\Rightarrow$ (ii)'' We write $\cup$ for the Tate pairing
  \[ H^1(k,E[n]) \times H^1(k,E[n]) \to \Br(k) \] given by the Weil
  pairing and cup product.  It is known (see \cite{ONeil,Zarhin}) that
  for any $\xi,\eta \in H^1(k,E[n])$ we have
  \[ \xi \cup \eta = \Ob_n(\xi + \eta) - \Ob_n(\xi) - \Ob_n(\eta). \]
  Compatibility with another pairing, also due to Tate (see for example
  \cite[Proposition 2.1]{CTPPS} or \cite[Section 5]{ONeil}), shows
  that $\xi \cup \eta = 0$ whenever $\xi$ and $\eta$ are both in the
  image of the connecting map $E(k)/nE(k) \to H^1(k,E[n])$.

  As in the statement of the proposition, we now take
  $\xi,\eta \in H^1(k,E[2])$ with $\Ob_2(\xi) = \Ob_2(\eta) = 0$.
  Then
  $\Ob_2(\xi+\eta) = \xi \cup \eta = \pi_*(\xi) \cup \pi_*(\eta) = 0$,
  where the first and last equality follow from the two facts we just
  quoted, and the middle equality holds since $\pi : E[2] \to F[2]$
  preserves the Weil pairing.  This must be the case since there is
  only one nondegenerate alternating pairing on $(\Z/2\Z)^2$.
  
  ``(ii) $\Rightarrow$ (i)'' We are given that
  $\Ob_2(\xi) = \Ob_2(\eta) = \Ob_2(\xi + \eta) = 0$. We may therefore
  represent $\xi$, $\eta$, $\xi+\eta$ by binary quartics $f$, $g$,
  $h$, all with the same invariants.  Our aim is to show that the
  $k$-points on $X_{f,g}$ are Zariski dense, since (i) then follows by
  Proposition~\ref{prop:use-pt}.  We are free to replace $f,g,h$ by
  $\lambda f, \lambda g, \lambda h$ for any $\lambda \in k^\times$, as
  this does not change $X_{f,g}$. We may therefore suppose that
  $y^2 = h(x,z)$ has a $k$-point, equivalently $y^2 = f(x,z)$ and
  $y^2 = g(x,z)$ are $k$-isomorphic as curves.  Let $C$ be this genus
  one curve, equipped with two different maps $C \to \PP^1$. For the
  same reasons as we explained at the end of the proof of
  Proposition~\ref{prop:use-pt}, we may ignore the case where $f$ and
  $g$ are equivalent. The image of the map $C \to \PP^1 \times \PP^1$
  is therefore defined by a $(2,2)$-form:
  \begin{equation*}
F(x_1,x_2;y_1,y_2) = \begin{pmatrix} x_1^2 & x_1 x_2 & x_2^2 \end{pmatrix}
\begin{pmatrix} 
a_{11} & a_{12} & a_{13} \\
a_{21} & a_{22} & a_{23} \\
a_{31} & a_{32} & a_{33} 
\end{pmatrix} \begin{pmatrix} y_1^2 \\ y_1y_2 \\ y_2^2 \end{pmatrix}.
\end{equation*}
We may compute this $(2,2)$-form by specialising the $(2,2,2)$-form in
\cite[Theorem~8.2]{bq-ctp} at a $k$-point on $y^2 = h(x,z)$.  The
binary quartics $f(x_1,x_2)$ and $g(y_1,y_2)$ are the discriminant of
$F$, when this is viewed as a binary quadratic form in the other set
of variables.  A short calculation reveals that
\[ X_{f,g} = \left\{ \begin{aligned} u_0 u_2 - u_1^2 &= v_0 v_2 - v_1^2 \\
l_2^2 - 4 l_1 l_3  &= m_2^2 - 4 m_1 m_3 \end{aligned} \right\} \subset \PP^5 \]
where
\[
\begin{pmatrix} l_1 \\ l_2 \\ l_3 \end{pmatrix}
= \begin{pmatrix} 
a_{11} & a_{21} & a_{31} \\
a_{12} & a_{22} & a_{32} \\
a_{13} & a_{23} & a_{33} 
\end{pmatrix} \begin{pmatrix} u_0 \\ u_1 \\ u_2 \end{pmatrix}
\,\, \text{ and } \,\,
\begin{pmatrix} m_1 \\ m_2 \\ m_3 \end{pmatrix}
= \begin{pmatrix} 
a_{11} & a_{12} & a_{13} \\
a_{21} & a_{22} & a_{23} \\
a_{31} & a_{32} & a_{33} 
\end{pmatrix} \begin{pmatrix} v_0 \\ v_1 \\ v_2 \end{pmatrix}.
\]
The rational parametrisation (a twisted form of~\eqref{ratparam})
\begin{align*}
  u_0 &= -a_{21} x^2 - a_{22} x y - a_{23} y^2 - 2 a_{31} x z - a_{32} (x t + y z) - 2 a_{33} y t, \\
  u_1 &= a_{11} x^2 + a_{12} x y + a_{13} y^2 - a_{31} z^2 - a_{32} z t - a_{33} t^2, \\
  u_2 &= a_{21} z^2 + a_{22} z t + a_{23} t^2 + 2 a_{11} x z + a_{12} (x t + y z) + 2 a_{13} y t, \\
  v_0 &= -a_{12} x^2 - a_{22} x z - a_{32} z^2 - 2 a_{13} x y - a_{23} (x t + y z) - 2 a_{33} z t, \\
  v_1 &= a_{11} x^2 + a_{21} x z + a_{31} z^2 - a_{13} y^2 - a_{23} y t - a_{33} t^2, \\
  v_2 &= a_{12} y^2 + a_{22} y t + a_{32} t^2 + 2 a_{11} x y + a_{21} (x t + y z) + 2 a_{31} z t,
\end{align*}
then shows that the $k$-points on $X_{f,g}$ are Zariski dense.
\end{proof}  

We now take $k$ a number field. If $\xi,\eta \in S^{(2)}(E/k)$ then a
standard argument, using the local-global principle for the Brauer
group, shows that $\Ob_2(\xi) = \Ob_2(\eta) = \Ob_2(\xi + \eta) =
0$. Combining Lemma~\ref{vislem} and Proposition~\ref{prop:ob}
therefore gives an alternative (more computational) proof of
Theorem~\ref{thm:main}.


\begin{thebibliography}{MM}
\frenchspacing
\renewcommand{\baselinestretch}{1}

\bibitem{AS}
A. Agashe and W.A. Stein,
Visibility of Shafarevich-Tate groups of abelian varieties,
{\em J. Number Theory} {\bf 97} (2002), no. 1, 171--185.

\bibitem{AS2}
A. Agashe and W.A. Stein,  
Visible evidence for the Birch and Swinnerton-Dyer conjecture
for modular abelian varieties of analytic rank zero,
With an appendix by J.E. Cremona and B. Mazur,
{\em Math. Comp.} {\bf 74} (2005), no. 249, 455--484.

\bibitem{AKM3P}
S.Y. An, S.Y. Kim, D.C. Marshall, S.H. Marshall, W.G. McCallum
and A.R. Perlis,
Jacobians of genus one curves,
{\em J. Number Theory} {\bf 90} (2001), no. 2, 304--315.

\bibitem{BCC}
B.S. Banwait, J. Caro and S. Chidambaram,
{\em On the visibility category of the Shafarevich--Tate group},
2026, \url{arXiv:2601.21519 [math.NT]}

\bibitem{BSDI}
B.J. Birch and H.P.F. Swinnerton-Dyer, 
Notes on elliptic curves, I.
{\em J. reine angew. Math.} {\bf{212}} (1963) 7--25.

\bibitem{magma}
W. Bosma, J. Cannon and C. Playoust,
The Magma algebra system, I. The user language,
{\em J. Symbolic Comput.}, {\bf{24}} (1997), 235--265. 

\bibitem{CaIV}
J.W.S. Cassels, 
Arithmetic on curves of genus 1, 
IV. Proof of the Hauptvermutung,
{\em J. reine angew. Math.} {\bf{211}} (1962) 95--112.

\bibitem{CSS}
J.-L. Colliot-Th\'el\`ene, J.-J. Sansuc and H.P.F. Swinnerton-Dyer,
Intersections of two quadrics and Châtelet surfaces, I.
{\em J. reine angew. Math.} {\bf 373} (1987), 37--107,
II. {\em J. reine angew. Math.} {\bf{374}} (1987), 72--168. 

\bibitem{CrTables}
J.E. Cremona, 
{\em Algorithms for modular elliptic curves},
Second edition, Cambridge University Press, Cambridge, 1997. 

\bibitem{Cr}
J.E. Cremona, 
Classical invariants and 2-descent on elliptic curves,
{\em J. Symbolic Comput.} {\bf{31}} (2001), no. 1-2, 71--87.

\bibitem{CFOSS}
J.E. Cremona, T.A. Fisher, C. O'Neil, D. Simon and M. Stoll,  
Explicit $n$-descent on elliptic curves, I. Algebra,
{\em J. reine angew. Math.} {\bf{615}} (2008), 121--155.

\bibitem{CF}
J.E. Cremona and T.A. Fisher, 
On the equivalence of binary quartics,
{\em J. Symbolic Comput.} {\bf{44}} (2009), no. 6, 673--682.

\bibitem{CM}
J.E. Cremona and B. Mazur,
Visualizing elements in the Shafarevich-Tate group,
{\em Experiment. Math.} {\bf 9} (2000), no. 1, 13--28.

\bibitem{CTPPS}
T.A. Fisher,
The Cassels-Tate pairing and the Platonic solids,
{\em J. Number Theory} {\bf 98} (2003), no. 1, 105--155.

\bibitem{invis}
T.A. Fisher,  
Invisibility of Tate-Shafarevich groups in abelian surfaces,
{\em Int. Math. Res. Not.} {\bf 2014}, no. 15, 4085--4099.

\bibitem{vis7}
T.A. Fisher,
Visualizing elements of order 7 in the Tate-Shafarevich group
of an elliptic curve,
{\em LMS J. Comput. Math.} {\bf 19} (2016), suppl. A, 100--114.

\bibitem{bq-ctp}
T.A. Fisher, 
On binary quartics and the Cassels-Tate pairing,
{\em Res. Number Theory} {\bf 8} (2022), no. 4, Paper No. 74, 13 pp.

\bibitem{GJPST}
G. Grigorov, A. Jorza, S. Patrikis, W.A. Stein and C. Tarniţǎ,  
Computational verification of the Birch and Swinnerton-Dyer
conjecture for individual elliptic curves,
{\em Math. Comp.} {\bf 78} (2009), no. 268, 2397--2425.

\bibitem{Klenke}
T.A. Klenke,  
Visualizing elements of order two in the Weil-Ch\^atelet group,
{\em J. Number Theory} {\bf 110} (2005), no. 2, 387--395.

\bibitem{Mazur}
B. Mazur, 
Visualizing elements of order three in the Shafarevich-Tate group,
{\em Asian J. Math.} {\bf 3} (1999), no. 1, 221--232.

\bibitem{4desc}
J.R. Merriman, S. Siksek and N.P. Smart,  
Explicit 4-descents on an elliptic curve,
{\em Acta Arith.} {\bf 77} (1996), no. 4, 385--404.

\bibitem{ONeil}
C. O'Neil,  
The period-index obstruction for elliptic curves,
{\em J. Number Theory} {\bf 95} (2002), no. 2, 329--339.

\bibitem{Reid}
M. Reid,
{\em The complete intersection of two or more quadrics},
PhD thesis, University of Cambridge, 1972.

\bibitem{Sal}
P. Salberger,
{\em On the arithmetic of intersections of two quadrics
containing a conic}, 1993, \url{arXiv:2305.02289 [math.NT]}

\bibitem{SD}
H.P.F. Swinnerton-Dyer, The Brauer group of cubic surfaces,
{\em Math. Proc. Cambridge Philos. Soc.} {\bf 113} (1993), no. 3, 449--460.

\bibitem{Weil54}
A. Weil, 
Remarques sur un m\'emoire d'Hermite,
{\em Arch. Math. (Basel)} {\bf{5}}, (1954) 197--202.

\bibitem{Zarhin}
Yu.G. Zarhin,
Noncommutative cohomology and Mumford groups,
{\em Mat. Zametki} {\bf 15} (1974), 415--419.

\end{thebibliography}
\end{document}